\documentclass[doublecolumn]{IEEEtran}
\usepackage{subfigure,amsmath,amsfonts,amssymb,color,graphicx, algorithmic,verbatim,mathrsfs,multirow,cite}

\newtheorem{theorem}{Theorem}
\newtheorem{lemma}{Lemma}

\newtheorem{corollary}{Corollary}
\newtheorem{assumption}{Assumption}

\newenvironment{proof}[1][Proof]{\begin{trivlist}
\item[\hskip \labelsep {\bfseries #1}]}{\end{trivlist}}

\newenvironment{remark}[1][Remark]{\begin{trivlist}
\item[\hskip \labelsep {\bfseries #1}]}{\end{trivlist}}

\newcommand{\qed}{\nobreak \ifvmode \relax \else
      \ifdim\lastskip<1.5em \hskip-\lastskip
      \hskip1.5em plus0em minus0.5em \fi \nobreak
      \vrule height0.75em width0.5em depth0.25em\fi}

\newlength{\minipagewidth}
\newcommand{\bookbox}[1]{
\par\medskip\noindent
\framebox[\columnwidth]{
\begin{minipage}{\minipagewidth}
{#1}
\end{minipage} } \par\medskip }

\def\deq{\triangleq}

\def\a{\alpha}

\def\t{\theta}

\def\Ec{\mathcal{E}}
\def\Fc{\mathcal{F}}
\def\Gc{\mathcal{G}}
\def\Nc{\mathcal{N}}
\def\Vc{\mathcal{V}}

\def\sX{\mathsf{X}}

\def\Reals{{\mathbb{R}}}

\def\1b{\mathbf{1}}
\def\0b{\mathbf{0}}

\def\norm#1{\left\|#1\right\|}
\def\abs#1{\left|#1\right|}
\def\paren#1{\left(#1\right)}

\def\la{\langle}
\def\ra{\rangle}
\def\rb{\mathbf{r}}
\def\eb{\mathbf{e}}

\def\zb{\mathbf{z}}

\def\xb{\mathbf{x}}
\def\yb{\mathbf{y}}

\def\ub{\mathbf{u}}

\def\fb{\mathbf{f}}

\def\Nin{\mathcal{N}_i^{\text{in}}}
\def\Nout{\mathcal{N}_i^{\text{out}}}

\def\tr{\textrm{tr}}

\def\bxb{\bar{\xb}}

\def\bzb{\bar{\zb}}
\def\bz{\bar{z}}

\def\an#1{{#1}}

\def\sml#1{{#1}}

\allowdisplaybreaks
\def\argmin{\operatornamewithlimits{arg\,min}}
\begin{document}
\title{{Coordinate Dual Averaging for Decentralized Online Optimization with
Nonseparable Global Objectives}*}
\author{Soomin Lee, Angelia Nedi\'{c}, and Maxim Raginsky
\thanks{*The work has been partially supported by
the National Science Foundation under grant no.~CCF 11-11342 and
by the Office of Naval Research under grant no.~N00014-12-1-0998.
}
\thanks{
S.~Lee and M.~Raginsky are with the Department of Electrical and Computer Engineering,
University of Illinois at Urbana-Champaign,  Urbana, IL, 61801 USA.
A.~Nedi\'c is with the Department of Industrial and Enterprise Systems Engineering,
University of Illinois at Urbana-Champaign,  Urbana, IL, 61801 USA.
Their emails are \tt{\{lee203,angelia,maxim\}@illinois.edu}.}
\thanks{
Portions of this work were presented at the 2015 American Control Conference and at the 2015 International Symposium on Mathematical Programming.
}
}

\maketitle
\begin{abstract}
We consider a decentralized online convex optimization problem in a network of agents, where each agent controls only a coordinate (or a part) of the global decision vector.
For such a problem, we propose two decentralized variants (\textbf{ODA-C} and \textbf{ODA-PS}) of Nesterov's primal-dual algorithm
with dual averaging.
In \textbf{ODA-C}, to mitigate the disagreements on the primal-vector updates,
the agents implement a generalization of the local information-exchange dynamics recently proposed by Li and Marden \cite{Li_Marden} over a static undirected graph. In \textbf{ODA-PS},
the agents implement the broadcast-based push-sum dynamics \cite{pushsum1} over a time-varying sequence of uniformly connected digraphs. We show that the regret bounds in both cases have sublinear growth of $O(\sqrt{T})$, with the time horizon $T$,
when the stepsize is of the form $1/\sqrt{t}$ and the objective functions are Lipschitz-continuous convex functions
with Lipschitz gradients.
We also implement the proposed algorithms on a sensor network to complement our theoretical analysis. \end{abstract}

\section{Introduction}
Decentralized optimization has recently been receiving significant attention
due to the emergence of large-scale distributed algorithms in machine learning, signal processing, and control applications for
wireless communication networks, power networks, and sensor networks; see, for example,
\cite{Bullo2009,Mesbahi2010,Kar2010,Martinoli2013,Zhang2013,HTC2014}.
A central generic problem in such applications is decentralized resource allocation
for a multiagent system, where the agents collectively solve an optimization problem
in the absence of full knowledge about the overall problem structure.
In such settings, the agents are allowed to communicate to each other some relevant estimates so as to learn
the information needed for an efficient global resource allocation.
The decentralized structure of the problem is reflected in the
agents' local view of the underlying communication network, where each agent exchanges
messages only with its neighbors.

In recent literature on control and optimization,
an extensively studied decentralized resource allocation problem is one where the system objective function $f(\xb)$
is given as a sum of local objective functions, i.e., $f(\xb)=\sum_{i=1}^n f_i(\xb)$ where $f_i$ is known only to agent $i$; see,
for example~\cite{Nedic2007,Johansson2007,Nedic09a,Ram2010,Kunal2011,rabbat_allerton2012,rabbat_cdc2012,NO,alex2014,Wei2012,Wei2013,Jakovetic2011a,Jakovetic2011b,Ling2014,Bahman-opt,Jorge-dist-online,next}.
In this case, the objective function is separable across the agents,
but the agents are coupled through the resource allocation vector $\xb$.
Each agent maintains and updates its own copy of the allocation/decision vector $\xb$,
while trying to estimate an optimal decision for the system problem. \sml{The vector $\xb$ is assumed to lie in (a subset of) $\Reals^d$, where $d$ may or may not coincide with the number of agents $n$.}

Another decentralized resource allocation problem is the one where the system objective function $f(\xb)$  may not admit a natural decomposition of the form $\sum^n_{i=1}f_i(\xb)$, and the resource allocation vector $\xb = (x_1,\ldots,x_n) \in \Reals^n$ is distributed among the agents, where each agent
$i$ is responsible for maintaining and updating only a coordinate (or a part) $x_i$ of the whole vector $\xb$.
Such decentralized problems have been considered in~\cite{johnthes,distasyn,TsitsiklisAthans,LiBasar1987,RKW11} (see also the textbook~\cite{distbook}).
In the preceding work, decentralized approaches converge when the agents are using weighted averaging, or
when certain contraction conditions are satisfied. Recently, Li and Marden \cite{Li_Marden} have proposed a different algorithm with local updates, where each agent
$i$ keeps estimates for the variables $x_j$, $j\ne i$, that are controlled by all the other agents in the network.
The convergence of this algorithm relies on some contraction properties of the iterates.
Note that all the aforementioned algorithms were developed for offline optimization problems.

\sml{Our work in this paper is motivated by the ideas of Li and Marden~\cite{Li_Marden} and also by the broadcast-based subgradient push, which
was originally developed by Kempe et al.~\cite{pushsum1} and later extended in \cite{pushsum2} and in \cite{rabbat_cdc2012,NO} to distributed optimization. Specifically,
we use the local information exchange model of~\cite{Li_Marden} and \cite{pushsum1,pushsum2,rabbat_cdc2012,NO}, but
employ a different online decentralized algorithm motivated by the work of Nesterov \cite{Nesterov_dual_averaging}. We call these algorithms \textbf{ODA-C} (Online Dual Averaging with Circulation-based communication) and \textbf{ODA-PS} (Online Dual Averaging with Push-Sum based communication), respectively.}

\sml{In contrast with existing methods,
our algorithms have the following distinctive features: (1) We consider an online convex optimization problem with nondecomposable system objectives, which are  functions of a distributed resource allocation vector.
(2) In our algorithms, each agent maintains and updates its private estimate of the best global allocation vector at each time, but contributes only one coordinate to the network-wide decision vector.
(3) We provide regret bounds in terms of the true global resource allocation vector $\xb$ (rather than some estimate on $\xb$ by a single agent).
For both \textbf{ODA-C} and \textbf{ODA-PS},
we show that the regret has sublinear growth of order $O(\sqrt{T})$ in time $T$ with the stepsize of the form $1/\sqrt{t+1}$.
}

Our proposed algorithm \textbf{ODA-PS} is closest to recent papers \cite{Queens,Mesbahi}.
The papers proposed a decentralized algorithm for online convex optimization which is very similar to \textbf{ODA-PS}
in a sense that they also introduce online subgradient estimations in primal \cite{Queens} or dual \cite{Mesbahi} space into information aggregation using push-sum.
In these papers, the agents share a common decision set in $\Reals^d$, the objective functions are separable across the agents at each time (i.e., $f_t(\xb) = \sum_{i=1}^n f_t^i(\xb)$ for all $t$), and the regret is analyzed in terms of each agent's own copy of the whole decision vector $\xb \in \Reals^d$. Moreover, an additional assumption is made in \cite{Queens} that the objective functions are strongly convex.

The paper is organized as follows.
In Section~\ref{sec:problem}, we formalize the problem
 and describe how the agents interact.
In Section~\ref{sec:basicregret}, we provide an online decentralized dual-averaging algorithm in a generic form and establish a basic regret bound which can be used later for particular instantiations, namely,  for the two algorithms \textbf{ODA-C} and \textbf{ODA-PS}. These algorithms are analyzed in  Sections~\ref{sec:algo-regret}, where we establish $O(\sqrt{T})$ regret bounds under mild assumptions. In Section~\ref{sec:sim}, we demonstrate our analysis by simulations on a sensor network.
We conclude the paper with some comments in Section~\ref{sec:conclusion}.

{\bf Notation:}
\an{All vectors are column vectors. For vectors associated with agent $i$, we use a subscript $i$ such as, for example, $\xb_i,\zb_i$, etc.
We will write $x^k_i$ to denote the $k$th coordinate value of a vector $\xb_i$.
We will work with the Euclidean norm, denoted by $\|\cdot\|$. We will use $\eb_1,\ldots,\eb_n$ to denote
the unit vectors in the standard Euclidean basis of $\mathbb{R}^n$.
We use $\1b$ to denote a vector with all entries equal to 1,
while $I$ is reserved for an identity matrix of a proper size.
For any $n \ge 1$, the set of integers $\{1,\ldots,n\}$ is denoted by $[n]$.
We use $\sigma_2(A)$ to denote the second largest singular value of a matrix $A$.
}

\section{Problem formulation}\label{sec:problem}

Consider a multiagent system (network) consisting of $n$ agents, indexed by elements of the set $\Vc = [n]$. Each agent $i \in \Vc$ takes actions in an action space $\sX$, which is a closed and bounded interval of the real line.\footnote{Everything easily generalizes to $\sX$ being a compact convex subset of a multidimensional space $\Reals^d$; we mainly stick to the scalar case for simplicity.}
\an{At each time, the multiagent system incurs a time-varying cost $f_t$, which
comes from a fixed class $\Fc$ of convex functions $f : \sX^n \to \Reals$. }

The communication among agents in the network is governed by either one of the two following models:
\begin{enumerate}
\item[(G1)] An undirected connected graph $\Gc = (\Vc,\Ec)$: If agents $i$ and $j$ are connected by an edge (which we denote by $i \leftrightarrow j$), then they may exchange information with one another. Thus, each agent $i \in \Vc$ may directly communicate only with the agents in its neighborhood
    $\Nc_i \deq \left\{ j \in \Vc : i \leftrightarrow j \right\} \cup \{i\}$.
    Note that agent $i$ is always contained in its own neighborhood.
\item[(G2)] Time-varying digraphs $\Gc(t) = (\Vc,\Ec(t))$, for $t \ge 1$: If there exists a directed link from agent $j$ to $i$ at time $t$ (which we denote by $(j,i)$), agent $j$ may send its information to agent $i$.
    We use the notation $\Nin(t)$ and $\Nout(t)$ to denote the in and out neighbors of agent $i$ at time $t$, respectively.
That is,
\[
\Nin(t) \deq \{j \mid (j,i) \in \Ec(t)\} \cup \{i\},
\]
\[
\Nout(t) \deq \{j \mid (i,j) \in \Ec(t)\} \cup \{i\}.
\]
In this case, we assume that there always exists a self-loop $(i,i)$ for all agent $i \in \Vc$. Therefore,
agent $i$ is always contained in its own neighborhood.
Also, we use $d_i(t)$ to denote the out degree of node $i$ at time $t$. i.e.,
\[
d_i(t) \deq |\Nout(t)|.
\]
We assume $B$-strong connectivity of the graphs $\Gc(t)$ with some scalar $B>0$, i.e., a graph with the following edge set
\[
\Ec_B(t) = \bigcup_{i=(t-1)B+1}^{tB} \Ec(i)
\]
is strongly connected for every $t \ge 1$.
\sml{In other words, the union of the edges appearing for $B$ consecutive time instances periodically
constructs a strongly connected graph. This assumption is required to ensure that
there exists a path from one node to every other node infinitely often
even if the underlying network topology is time-varying.}
\end{enumerate}

The network interacts with an environment according to the protocol shown in Figure~\ref{fig:OGO}. We leave the details of the signal generation process vague for the moment, except to note that the signals received by all agents at time $t$ may depend on all the information available up to time $t$ (including $f_1,\ldots,f_t$, as well as all of the local information exchanged in the network). Moreover, the environment may be adaptive, i.e., the choice of the function $f_t$ may depend on all of the data generated by the network up to time $t$.

\begin{figure}[ht]
\bookbox{\small
{\textsf{Parameters:}} base action space $\sX$; network graph $\Gc = (\Vc,\Ec)$; function class $\Fc$

\medskip\noindent
For each round $t=1,2,\ldots$:
\begin{itemize}
\item[(1)] Each agent $i \in \Vc$ selects an action $x^i(t) \in \sX$
\item[(2)] Each agent $i \in \Vc$ exchanges local information with its neighbors $\Nc_i$
\item[(3)] The environment selects the current objective $f_t \in \Fc$, and each agent receives a signal about $f_t$
\end{itemize}
}
\caption{Online optimization with global objectives and local information.}
\label{fig:OGO}
\end{figure}

Let us denote the network action at time $t$ by
\begin{align}\label{eq:net-ac}
\xb(t) = (x^1(t),\ldots,x^n(t)) \in \sX^n.
\end{align}
We consider the \textit{network regret} \an{$R(T)$ at an arbitrary time horizon $T\ge1$}:
\begin{align}\label{eq:regret}
	R(T) \deq \sum^T_{t=1} f_t(\xb(t)) - \inf_{\yb \in \sX^n}\sum^T_{t=1}f_t(\yb).
\end{align}
\an{Thus, $R(T)$} is the difference between the total cost incurred by the network at time $T$
and the smallest total cost that could have been achieved with a single action in $\sX^n$ in hindsight
(i.e., with perfect advance knowledge of the sequence $f_1,\ldots,f_T$) and
without any restriction on the communication between the agents.
The problem is to design the rule (or policy) each agent $i \in \Vc$ should use to determine its action $x^i(t)$
based on the local information available to it at time $t$, such that the regret in \eqref{eq:regret} is
(a) sublinear as a function of the time horizon $T$ and
(b) exhibits ``reasonable'' dependence on the number of agents $n$ and on the topology of the communication graphs.

\sml{
The regret in (\ref{eq:regret}) is defined over the true network actions of individual agents, i.e., $x^i(t)$'s, rather than in terms of some estimates of $\mathbf{x}(t)$ by individual agents. This notion of regret, which, to the best of our knowledge has been first introduced in \cite{RKW11}, is inspired by the literature on team decision theory and decentralized control problems: The online optimization is performed by a \textit{team} of cooperating agents facing a time-varying sequence of global objective functions $f_t$, which are nondecomposable (in contrast to decomposable objectives $\sum_i f_t^i(\mathbf{x})$, where $f_t^i$ is only revealed to agent $i$). Communication among agents is local, as dictated by the network topology, so  no agent has all the information in order to compute a good global decision vector $\xb(t)$. By comparing the cumulative performance of the decentralized system to the best \textit{centralized} decision achievable in hindsight, the regret in \eqref{eq:regret} captures the effect of decentralization. It also calls for analysis techniques that are different from existing methods in the literature.
}

\section{The basic algorithm and regret bound}\label{sec:basicregret}

We now introduce a generic algorithm for solving the decentralized online optimization problem defined in Section \ref{sec:problem}. The algorithm uses the dual-averaging subgradient method of Nesterov \cite{Nesterov_dual_averaging} as an optimization subroutine.

\sml{Each agent $i \in \Vc$ generates a sequence $\{ \xb_i(t),\zb_i(t) \}^\infty_{t=1}$ in $\sX^n \times \Reals^n$, where the primal iterates
$$
\xb_i(t) = (x^1_i(t),\ldots,x^n_i(t)) \in \sX^n
$$
and the dual iterates
$$
\zb_i(t) = (z^1_i(t),\ldots,z^n_i(t)) \in \Reals^n
$$
are updated recursively as follows:
\begin{subequations}
\begin{align}
z_i^k(t+1)
& = \frac{1}{r_i}\delta_i^k u_i(t) + F^k_{i,t}\left({\mathsf m}_i(t)\right),~ k \in [n] \label{eqn:gdyn1}\\
\xb_i(t+1)
& = \Pi_{\sX^n}^{\psi} \left(G_{i,t}(\zb_i(t+1)),\alpha(t)\right)\label{eqn:gdyn2}
\end{align}
\end{subequations}
with the initial condition $\zb_i(0) = 0$ for all $i \in \Vc$. In the dual update \eqref{eqn:gdyn1}, $\delta_i^k$ is the Kronecker delta symbol, $r_i > 0$ is a positive weight parameter, $u_i(t) \in \mathbb{R}$ is a local update computed by agent $i$ at time $t$ based on the received signal about $f_t$, ${\mathsf m}_i(t)$ are the messages received by agent $i$ at time $t$ [from agents in $\Nc_i$ under the model (G1) or from $\Nc_i^{\text{in}}(t)$ under the model (G2)], and $F^k_{i,t}$, $k \in [n]$, are real-valued mappings that perform local averaging of ${\mathsf m}_i(t)$. In the primal update \eqref{eqn:gdyn2}, $G_{i,t} : \Reals^n \to \Reals^n$ is a mapping on dual iterates, $\{\alpha(t)\}^\infty_{t=0}$ is a nonincreasing sequence of positive step sizes, and the mapping $\Pi^\psi_{\sX^n} : \Reals^n \times (0,\infty) \to \sX^n$ is defined by
\begin{align}
	\Pi_{\sX^n}^{\psi} (\zb,\alpha) \deq \argmin_{\xb \in \sX^n} \left\{\la \zb, \xb \ra + \frac{1}{\alpha}\psi(\xb)\right\},
\end{align}
where $\psi : \sX^n \to \Reals^+$ is
a nonnegative \textit{proximal function}. We assume that $\psi$ is $1$-strongly convex with respect to the Euclidean norm $\| \cdot \|$, i.e., for any $\xb,\yb \in \sX^n$ we have
\begin{align}
	\psi(\yb) \ge \psi(\xb) + \la \bar{\nabla}\psi(\xb),\yb - \xb \ra + \frac{1}{2} \| \xb - \yb \|^2,
\end{align}
where $\bar{\nabla}\psi$ denotes an arbitrary subgradient of $\psi$.}

\sml{
The dual iterate $\zb_i(t)$ computed by agent $i$ at time $t$ will be an estimate of the ``running average of the subgradients'' as seen by agent $i$,
and will constitute an approximation of the true centralized dual-averaging subgradient update of Nesterov's algorithm.
The messages from $\Nc_i$ entering into the dual-space dynamics
are crucial for mitigating any disagreement between the agents' local estimates of what the network action should be.
The primal iterate $\xb_i(t)$ of agent $i$ at time $t$ is an approximation of
the true centralized primal point for the subgradient evaluation.
}

\sml{Note that in \eqref{eqn:gdyn1} the local update $u_i(t)$ based on the signal about $f_t$ affects
affects only the $i$th coordinate of the dual iterate $\zb_i(t+1)$,
while all other coordinates with  $k \neq i$ remain untouched except for the averaging.
The action of agent $i$ at time $t$ is then given by
\begin{align*}
	x^i(t) = x^i_i(t),
\end{align*}
i.e., by the $i$th component of the vector $\xb_i(t)$.}

\sml{A concrete realization of the algorithm (\ref{eqn:gdyn1})-(\ref{eqn:gdyn2}) requires specification of the rules for computing the local update $u_i(t)$, the messages exchanged by the agents, and the mappings $F^k_{i,t}$ and $G_{i,t}$. In this paper, we present two different instantiations of this algorithm, namely, the circulation-based method inspired by \cite{Li_Marden} and the push-sum based method inspired by \cite{pushsum1,pushsum2,rabbat_cdc2012,NO}.
We call these algorithms \textbf{ODA-C} (Online Dual Averaing with Circulation-based communication) and \textbf{ODA-PS} (Online Dual Averaing with Push-Sum based communication)
and detail them in Section \ref{sec:algo-regret} and \ref{sec:algo-regret2}, respectively.
}

We now present a basic regret bound that can be used for any generic algorithm of the form (\ref{eqn:gdyn1})-(\ref{eqn:gdyn2}) under the following assumption:
\begin{assumption}\label{assume:Lipschitz}
All functions $f \in \Fc$ are Lipschitz continuous with a constant $L$:
\begin{align*}
	| f(\xb) - f(\yb) | \le L \| \xb - \yb \|\quad \hbox{for all }\xb,\yb \in \sX^n.
\end{align*}
\end{assumption}

\begin{theorem}\label{thm:main} Let $\{\xb_i(t)\}^\infty_{t=1} \subset \sX^n$, $i \in \Vc$, be the sequences of the agents' primal iterates, let $\{\ub(t)\}^\infty_{t = 1}$ with $\ub(t)=(u_1(t),\ldots,u_n(t))$ be the sequence of the agents' local updates, and let $\{\bar\xb(t)\}^\infty_{t= 1} \subset \sX^n$ be generated as
\begin{equation}\label{eq:xu}
\bar\xb(t+1) = \Pi_{\sX^n}^{\psi}\left(\sum_{s=0}^t \ub(s),\a(t)\right).
\end{equation}
Then, under Assumption \ref{assume:Lipschitz},
the \textit{network regret} $R(T)$ in (\ref{eq:regret}) can be upper-bounded in terms of $\ub(t)$ and $\bxb(t)$ as \an{follows: for each $T\ge1$,}
	\begin{align*}
		&\an{R(T)} \le  \underbrace{\frac{1}{2}\sum^T_{t=1} \alpha(t-1) \| \ub(t) \|^2}_{\textrm{(E1)}} + \frac{C}{\alpha(T)} \nonumber \\
		& + \underbrace{L \sum^T_{t=1} \sum^n_{i=1} \| \xb_i(t) - \bxb(t) \|}_{\textrm{(E2)}}
		+ \underbrace{\sqrt{n} D_\sX \sum^T_{t=1} \| \nabla f_t(\bxb(t)) - \ub(t) \|}_{\textrm{(E3)}},
	\end{align*}
where $D_\sX \deq \sup_{x,y \in \sX}|x-y|$ is the diameter of the set~$\sX$, and $C \deq \sup_{x \in \sX^n}|\psi(x)|$.
\end{theorem}
\begin{remark} Since $\psi$ is a continuous function on the compact set $\sX^n$, $C < \infty$ by the Weierstrass theorem.\end{remark}

\begin{proof}
\an{For any $t$ and any $\yb \in \sX^n$} we can write
	\begin{align}
		&f_t(\xb(t)) - f_t(\yb) \nonumber\\
		&= f_t(\xb(t)) - f_t(\bxb(t)) + f_t(\bxb(t)) - f_t(\yb)\nonumber \\
		&\le \la \nabla f_t(\xb(t)), \xb(t) - \bxb(t) \ra + \la \nabla f_t (\bxb(t)), \bxb(t) - \yb \ra \nonumber \\
		&\le L \| \xb(t) - \bxb(t) \| + \la \nabla f_t (\bxb(t)), \bxb(t) - \yb \ra,
		 \label{eqn:regret_bound_1}
	\end{align}
	where the second step follows from convexity of $f_t$, while the last step uses the fact that all $f \in \Fc$ are $L$-Lipschitz.
	\an{Recalling that $\xb(t)$ is the network action vector (see~\eqref{eq:net-ac}),}
	we have the following for the first term in \eqref{eqn:regret_bound_1}:
	\begin{align}
		\| \xb(t) - \bxb(t) \| &= \left\| \sum^n_{i=1} \left(x^i(t) - \bar{x}^i(t) \right) \eb_i \right\| \nonumber \\
		&\le \sum^n_{i=1} \| \xb_i(t) - \bxb(t) \|, \label{eq:eqn_regret_bound_1a}
	\end{align}
	where the equality follows from the definition of $\xb(t)$ in (\ref{eq:net-ac}) and $\bxb(t) = (\bar x^1(t),\ldots,\bar x^n(t))$.
	
	 The second term in~\eqref{eqn:regret_bound_1} can be further expanded as
	\begin{align}
	&	\la \nabla f_t (\bxb(t)), \bxb(t) - \yb \ra \nonumber\\
	&= \la \ub(t), \bxb(t) - \yb \ra + \la \nabla f_t(\bxb(t)) - \ub(t), \bxb(t) - \yb \ra. \label{eqn:regret_bound_2}
	\end{align}
Now, from relation (\ref{eq:xu}) we obtain
\begin{align}
	\bxb(t+1)
	&= \argmin_{\xb \in \sX^n} \left\{ \sum^t_{s=0} \la \ub(s), \xb \ra + \frac{1}{\alpha(t)}\psi(\xb)\right\}. \nonumber
\end{align}
Therefore, by \cite[Lemma~3]{Duchi_etal_dis_ave}, we can write
\begin{align}\label{eqn:regret_bound_3}
	\sum^T_{t=1} \la \ub(t), \bxb(t) - \yb \ra \le \frac{1}{2}\sum^T_{t=1} \alpha(t-1) \| \ub(t) \|^2 + \frac{\psi(\yb)}{\alpha(T)}.\quad
\end{align}
For the second term on the right-hand side of \eqref{eqn:regret_bound_2}, we have
\begin{align}
&	\la \nabla f_t(\bxb(t)) - \ub(t), \bxb(t) - \yb \ra \nonumber\\
&\qquad \le \| \bxb(t) - \yb \| \| \nabla f_t(\bxb(t)) - \ub(t) \| \nonumber \\
&\qquad \le \sqrt{n} D_\sX \| \nabla f_t(\bxb(t)) - \ub(t) \|. \label{eqn:regret_bound_4a}
\end{align}
Combining the estimates in Eqs.~\eqref{eqn:regret_bound_1}-\eqref{eqn:regret_bound_4a} and taking the supremum over all $\yb \in \sX^n$, we get the desired result.~$\square$
\end{proof}

\noindent\sml{Theorem~\ref{thm:main} indicates that the regret will be small provided that
\begin{itemize}
\item[(E1)] The squared norms $\| \ub(t) \|^2$ remain bounded.
\item[(E2)] The agents' primal variables $\xb_i(t)$ do not drift too much from the centralized vector $\bxb(t)$.
\item[(E3)] The vectors $\ub(t)$ stay close to the gradients $\nabla f_t(\bxb(t))$.
\end{itemize}
This theorem plays an important role in the sequel, since it provides guidelines for designing the update rule $u_i(t)$ and the mappings $F^k_{i,t}(\cdot)$ and $G_{i,t}(\cdot)$. We will also see later that the centralized vector $\bar\xb(t)$ represents a ``mean field'' of the primal iterates $\xb_i(t)$ for $i \in \Vc$ at time $t$.}

\section{\textbf{ODA-C} and its regret bound}\label{sec:algo-regret}
We now introduce a decentralized online optimization algorithm which uses
a circulation-based framework 
for its dual update rule \eqref{eqn:gdyn1}. We refer to this algorithm as \textbf{ODA-C} (Online Dual Averaing with Circulation-based communication).
\textbf{ODA-C} uses the network model (G1) for its communication.

\subsection{\textbf{ODA-C}}
Fix a vector $\rb = (r_1,\ldots,r_n)$ of positive weights and a nonnegative $n \times n$ matrix $M$, such that $M_{ij} \neq 0$ only if $j \in \Nc_i$, satisfying the following symmetry condition:
\begin{align}\label{eq:symmetry}
	r_i M_{ij} = r_j M_{ji}, \qquad i,j \in \Vc.
\end{align}
Then, \textbf{ODA-C} uses the following instantiation of the update rules in (\ref{eqn:gdyn1})-(\ref{eqn:gdyn2}):
\begin{subequations}
\begin{align}
z_i^k(t+1)
& = \frac{1}{r_i} \delta_i^k u_i(t) + z_i^k(t)  \nonumber\\
& \quad  + \sum^n_{j=1} M_{ij}\left( v_{j\to i}^k (t) -  v_{i\to j}^k (t)\right), \,\, k \in [n]\label{eqn:dyn1}\\
\xb_i(t+1)
& = \Pi_{\sX^n}^{\psi} \left(\zb_i(t+1),\alpha(t)\right),\label{eqn:dyn2}
\end{align}
\end{subequations}
where $(v^1_{j \to i}(t),\ldots,v^n_{j \to i}(t)) \in \Reals^n$ represents a vector of messages transmitted by agent $j$ to agent $i$, provided that $j \in \Nc_i$. \sml{Since $i \in \Nc_i$, we may include the previous dual iterate $\zb_i(t)$ and the outgoing messages $v^k_{i \to j}(t)$ in ${\mathsf m}_i(t)$.} The dual update rule \eqref{eqn:dyn1} is inspired by the state dynamics proposed by Li and Marden \cite{Li_Marden}, whereas the primal update rule \eqref{eqn:dyn2} is exactly what one has in Nesterov's scheme~\cite{Nesterov_dual_averaging}.

To complete the description of the algorithm, we must specify the update policies $\{u_i(t)\}$ and the messages $\{v^k_{i \to j}(t)\}$. We assume that all agents receive a complete description of $f_t$. Agent $i$ then computes
\begin{align}\label{eqn:gradient_estimate}
		u_i(t) &= \la \nabla f_t(\xb_i(t)), \eb_i \ra,~ i \in [n],~t\ge 0.
\end{align}
and  feeds this signal
back into the dynamics (\ref{eqn:dyn1}). Note, however, that the execution of the algorithm will not change if the agents never directly learn the full function $f_t$, nor even the full gradient $\nabla f_t(\xb(t))$, but instead receive the local gradient signal $\nabla f_t(\xb_i(t))$. The messages $v_{i\to j}(t)$ take the form
\begin{align}
	& v^k_{i \to j}(t) =  z^k_i(t)\label{eqn:Markov_policy}
\end{align}
for all $t$ and all agents $i,j \in \Vc$ \an{with $j \in \Nc_i$}.

\subsection{Regret of \textbf{ODA-C} with local gradient signals}
Let $\bar{\zb}(t) = (\bar{z}^1(t),\ldots,\bar{z}^n(t)).$ Our regret analysis rests on the following simple but important fact:
\begin{lemma}\label{lm:mean_field}
The weighted sum
\begin{align*}
	\bzb(t) \deq \sum^n_{i=1} r_i \zb_i(t)
\end{align*}
evolves according to the linear dynamics
\begin{align}\label{eqn:zbar}
	\bzb(t+1) = \bzb(t) + \ub(t),
\end{align}
where $\ub(t) = \big( u_1(t),\ldots,u_n(t)\big)$.
\end{lemma}
\begin{remark}
We observe that the relation in~\eqref{eqn:zbar} holds regardless of the choices of decisions
$v^k_{j\to i}(t)$ and $v^k_{i\to j}(t)$.
Moreover, we point out that if $\ub(t) = \nabla f_t(\xb(t))$, then the combination of \eqref{eqn:zbar} and \eqref{eqn:dyn2} will reduce to a centralized online variant of Nesterov's scheme~\cite{Xiao_RDA}.
\end{remark}
\begin{proof} Let $V^k(t)$ denote the $n \times n$ matrix with entries $[V^k(t)]_{ij} = v^k_{j\to i}(t) - v^k_{i \to j}(t)$. Then
\begin{align*}
	\bz^k(t+1) &= \sum^n_{i=1}r_i z^k_i(t+1) \\
	&= \sum^n_{i=1} r_i \left\{ z^k_i(t) + \frac{1}{r_i} \delta^k_i u_i(t) + \sum^n_{j=1}M_{ij} [V^k(t)]_{ij}\right\}
 \\
	&= \bz^k(t) + u_k(t) + \tr [\tilde{M}V^k(t)],
\end{align*}
where $\tilde{M}$ is an $n\times n$ matrix with entries $\tilde{M}_{ij} = r_i M_{ij}$. Since $\tilde{M}$ is a symmetric matrix, by \eqref{eq:symmetry}, and $V^k(t)$ is skew-symmetric, $\tr [\tilde{M} V^k(t)] = 0$, so we obtain \eqref{eqn:zbar}.~$\square$
\end{proof}

\noindent\sml{Lemma \ref{lm:mean_field} indicates that the vector $\bzb(t)$ can be seen as a ``mean field''
of the local dual iterates $\zb_i(t)$ for $i \in \Vc$ at time $t$.
Also, if we define
\[
\bxb(t+1) \deq \Pi^\psi_{\sX^n}(\bzb(t+1),\alpha(t)),
\]
then from relation (\ref{eqn:zbar}) we have
\[
\bxb(t+1) = \Pi^\psi_{\sX^n}\paren{\sum_{s=0}^t\ub(s),\alpha(t)},
\]
which coincides with relation (\ref{eq:xu}) in Theorem~\ref{thm:main}.
This allows us to make use of Theorem~\ref{thm:main} in analyzing the regret of this algorithm. Furthermore, the definition of $\bar{\xb}(t)$ and relation \eqref{eqn:gradient_estimate}
indicate that $\ub(t)$ will stay close to the centralized gradient $\nabla f_t(\bar{\xb}(t))$,
and as a consequence, the errors (E1) and (E3) in Theorem~\ref{thm:main} will remain small.
}

We now particularize the bound in Theorem \ref{thm:main} to this scenario under the following additional assumption:
\begin{assumption}\label{assume:Lipgrad}
All functions $f \in \Fc$ are differentiable and have Lipschitz continuous gradients with constant $G$:
\begin{align*}
	\| \nabla f(\xb) - \nabla f(\yb) \| \le G \| \xb - \yb \|, \quad \forall f \in \Fc;\, \xb,\yb \in \sX^n.
\end{align*}
\end{assumption}

\begin{theorem}\label{thm:local_grad_signals}
Under Assumptions~\ref{assume:Lipschitz}--\ref{assume:Lipgrad}, the regret of any algorithm of the form (\ref{eqn:dyn1})-(\ref{eqn:dyn2}), and with $\ub(t)$ computed according to \eqref{eqn:gradient_estimate}, can be upper-bounded as follows:
	\begin{align*} 
		\an{R(T)} & \le  \frac{nL^2}{2}\sum^T_{t=1} \alpha(t-1) + \frac{C}{\alpha(T)} \nonumber\\
		& \quad + \left(L + \sqrt{n}GD_\sX\right)\sum^T_{t=1} \alpha(t-1) \sum^n_{i=1} \| \zb_i(t) - \bzb(t) \|.
	\end{align*}
\end{theorem}

\begin{proof}
The terms on the right-hand side of the bound in Theorem \ref{thm:main} can be further estimated as follows.
Since each $f_t \in \Fc$ is $L$-Lipschitz,
	\begin{align*}
		\| \ub(t) \|^2 &= \sum^n_{i=1} \left| \la \nabla f_t(\xb_i(t)), \eb_i \ra\right|^2 \\
		& \le \sum^n_{i=1} \| \nabla f_t(\xb_i(t)) \|^2 \\
		& \le nL^2.
	\end{align*}
It remains to estimate term (E3) in Theorem \ref{thm:main}. To that end, we write
	\begin{align*}
	&\| \nabla f_t(\bxb(t)) - \ub(t) \|	\\
	&\quad=  \left\| \sum^n_{i=1} \la \nabla f_t(\bxb(t)) - \nabla f_t(\xb_i(t)), \eb_i \ra \eb_i \right\| \nonumber \\
		&\quad\le \sum^n_{i=1} \| \nabla f_t(\bxb(t)) - \nabla f_t(\xb_i(t)) \| \nonumber \\
		&\quad \le G \sum^n_{i=1} \| \bxb(t) - \xb_i(t) \|,
	\end{align*}
	where we have exploited the fact that the gradients of all $f \in \Fc$ are $G$-Lipschitz.

Now, by construction,
	\begin{align*}
		&\| \bxb(t) - \xb_i(t) \| \\
		&\quad = \left\| \Pi^\psi_{\sX^n}(\bzb(t),\alpha(t-1)) - \Pi^\psi_{\sX^n}(\zb_i(t),\alpha(t-1)) \right\| \nonumber \\
		&\quad \le \alpha(t-1) \| \bzb(t) - \zb_i(t) \|,\label{eqn:regret_bound_4}
 	\end{align*}
where the last step follows from the fact that the map $\zb \mapsto \Pi^\psi_{\sX^n}(\zb,\alpha)$ is $\alpha$-Lipschitz
	(see, e.g., \cite[Lemma~1]{Nesterov_dual_averaging}).
	Substituting these estimates into the bound in Theorem \ref{thm:main}, we get the result. $\square$
\end{proof}
\sml{This bound indicates that, if the network-wide disagreement term behaves nicely, the regret $R(T)$ will be sublinear in $T$ with a proper choice of the step size $\a(t)$.
We illustrate this more specifically in the following corollary.}

\begin{corollary}\label{cor:regret_bound} Suppose that the policies for computing $\{u_i(t)\}$
and $\{ v^k_{i \to j}(t)\}$ are such that, for all $t$ and for any sequence $f_1,\ldots,f_T \in \Fc$,
	\begin{align*}
		\sum^n_{i=1} \| \zb_i(t) - \bzb(t) \| \le K
	\end{align*}
	for some finite constant $K > 0$ (which may depend on $n$ and on other problem parameters). Then,
	the regret of the algorithm (\ref{eqn:dyn1})-(\ref{eqn:dyn2}) is bounded by
	\begin{align*}
		\an{R(T)} \le \left[\frac{nL^2}{2} + K\left(L + \sqrt{n}GD_\sX\right)\right]\sum^T_{t=1}\alpha(t-1) + \frac{C}{\alpha(T)}.
	\end{align*}
\end{corollary}
\noindent In particular, if we choose $\alpha(t) = \frac{1}{\sqrt{t+1}}$ for $t \ge 0$, then
the regret is \an{of the order} $O(\sqrt{T})$:
\begin{align*}
	R(T) \le \left[nL^2 + 2K\left(L + \sqrt{n}GD_\sX\right)\right] \sqrt{T} + C\sqrt{T+1} .
\end{align*}

\subsection{Full regret analysis}
\label{ssec:linear_ave}

\sml{We now show that the network-wide disagreement term is indeed upper-bounded by some constant. We recall that $M_{ij} \neq 0$ only if $j \in \Nc_i$. In addition to this, we posit the following assumptions on the pair $(\mathbf{r},M)$.
\begin{assumption}\label{assume:M}
The positive weights $r_1,\ldots,r_n$ sum to one:
\[
\sum_{i=1}^n r_i = 1 \textrm{ and } r_i > 0 \text{ for each } i \in [n].
\]
The matrix $M$ is row-stochastic, i.e.,
\[
\sum^n_{j=1}M_{ij}=1  \text{ for each } i \in [n].
\]
\end{assumption}
}
\noindent The conditions we have imposed on the pair $(\rb,M)$ are equivalent to saying that $M$ is the transition probability matrix of a reversible random walk on $\Gc$ with invariant distribution $\rb = (r_1,\ldots,r_n)$ \cite{Levin_Peres_Wilmer}.
\an{Let
\begin{align}
\zb^k(t)=(z^k_1(t),\ldots,z^k_n(t)),~k\in[n],~t\ge 0,\label{eq:zkt}
\end{align}
and $r_* \deq \min_{1 \le i \le n}r_i$. We state the following bound for $\sum_{i=1}^n \|\zb_i(t) - \bzb(t)\|^2$:}
\begin{lemma}\label{lem:disagree} \an{Under Assumptions~\ref{assume:Lipschitz} and \ref{assume:M},}
for the policy in \eqref{eqn:gradient_estimate}-\eqref{eqn:Markov_policy} we have
		\begin{align*}
\sum_{i=1}^n \|\zb_i(t) - \bzb(t)\|^2 \le \frac{n L^2}{r^3_*(1-\sqrt{1-\lambda})^2}
	\end{align*}
for every $t \ge 1$,
where
$$
\| \fb \|_\rb \deq \sqrt{ \sum^n_{i=1}r_i f_i^2}
$$
is \an{the $\rb$-weighted $\ell_2$-norm} of the vector $\fb \in \Reals^n$, and
where $\lambda$ denotes the spectral gap of $M$ \cite{Levin_Peres_Wilmer}, i.e.,
\begin{align*}
	\lambda &= \inf_{\fb \in \Reals^n,\, \la \rb, \fb \ra = 0} \frac{\| \fb \|^2_\rb - \| M \fb \|^2_\rb}{\| \fb \|^2_\rb}.
\end{align*}
\end{lemma}

\begin{proof} From \an{the definitions of $\zb_i(t), \bzb(t),$ and $\zb^k(t)$, we have}
\begin{align}\label{eqn:twoway}
\sum_{i=1}^n \|\zb_i(t) - \bzb(t)\|^2 = \sum_{k=1}^n \|\zb^k(t)-\bz^k(t)\1b\|^2.
\end{align}
Thus, we upper-bound the quantity on the right-hand side.

From \eqref{eqn:Markov_policy}, we can rewrite the dynamics \eqref{eqn:dyn1} as \an{follows:}
	\begin{align}\label{eqn:zvec}
		\zb^k(t+1) &= M\zb^k(t) + \frac{1}{r_k}u_k(t) \eb_k,
	\end{align}
	\an{where $\zb^k(t)$ is defined in~\eqref{eq:zkt}.}
	By unrolling the dynamics \eqref{eqn:zvec} and \eqref{eqn:zbar} from time 0 to $t$ and recalling that $\zb_i(0) = 0$ for all $i$, we obtain:
	\begin{align}\label{eqn:zvec2}
	\zb^k(t) = \frac{1}{r_k}\sum_{s=0}^{t-1} M^{t-s-1} u_k(s)\eb_k.
	\end{align}
	\an{Moreover, by the definition of $\bzb(t)$ in Eq.~\eqref{eqn:zbar},} we have
	\begin{align}\label{eqn:zslr}
	\bar{z}^k(t) =  \frac{1}{r_k}\sum_{s=0}^{t-1} r_k u_k(s).
	\end{align}
	Note that $r_k = \la \rb, \eb_k \ra$. From (\ref{eqn:zvec2}) and (\ref{eqn:zslr}), we have
	\begin{align}\label{eqn:dyn}
&	\|\zb^k(t) - \bz^k(t)\1b\| \le  \frac{1}{r_k}\sum_{s=0}^{t-1}\left\|M^{t-s-1} \eb_k - \la \rb, \eb_k \ra \1b \right\||u_k(s)|.
	\end{align}

By the properties of Markov matrices~\cite{Levin_Peres_Wilmer}, for any $\fb \in \Reals^n$,
\begin{align*}
	\left\| M^t \fb - \la \rb,\fb \ra\1b \right\|^2 &\le \frac{1}{r_*} \left\| M^t \fb - \la \rb, \fb \ra \1b \right\|^2_\rb \nonumber\\
	&\le \frac{(1-\lambda)^t}{r_*} \| \fb - \la \rb, \fb \ra \1b \|^2_\rb.
\end{align*}
Therefore,
	\begin{align}\label{eqn:right}
	&\left\|M^{t-s-1} \eb_k - \la \rb, \eb_k \ra \1b\right\|^2 \nonumber\\
	&\qquad\le  \frac{(1-\lambda)^{t-s-1}}{r_*} \left\| \eb_k - \la \rb, \eb_k \ra \1b \right\|^2_\rb \nonumber\\
	&\qquad = \frac{(1-\lambda)^{t-s-1}}{r_*} r_k (1-r_k) \nonumber \\
	&\qquad \le \frac{(1-\lambda)^{t-s-1}}{r_*}.
	\end{align}
	From relations (\ref{eqn:dyn}) and (\ref{eqn:right}), we obtain
	\begin{align*}
	\|\zb^k(t) - \bz^k(t)\1b\|^2
	&   \le  \left(\frac{1}{r^{3/2}_*} \sum_{s=0}^{t-1} (1-\lambda)^{\frac{t-s-1}{2}} |u_k(s)|\right)^2\\
	& \le  \frac{L^2}{r^{3}_*(1-\sqrt{1-\lambda})^2},
	\end{align*}
	where Assumption~\ref{assume:Lipschitz} is used in the last inequality.
	From this and relation \eqref{eqn:twoway}, we obtain
	\begin{align*}
	\sum_{i=1}^n \|\zb_i(t) - \bzb(t)\|^2
\le
	\frac{n L^2}{r^3_*(1-\sqrt{1-\lambda})^2},
\end{align*}
	\an{which proves the stated result}.~$\square$ \end{proof}

\noindent\sml{Lemma \ref{lem:disagree} captures the effect of the underlying network topology via the spectral gap $\lambda$ (also known as the Fiedler value), which captures the algebraic connectivity of the network. Since $\Gc$ is assumed to be connected, $\lambda > 0$.}

By combining Theorem \ref{thm:local_grad_signals} and Lemma \ref{lem:disagree}, we can now provide a regret bound for \textbf{ODA-C}:
\begin{theorem}\label{thm:linear_ave} \an{Let Assumptions~\ref{assume:Lipschitz}--\ref{assume:M} hold.
With the choice $\alpha(t) = \frac{1}{\sqrt{t+1}}$ for all $t\ge0$,
and under the policy \eqref{eqn:gradient_estimate}-\eqref{eqn:Markov_policy}, the distributed algorithm \textbf{ODA-C} achieves the following regret:}
	\begin{align*}
		R(T) &\le nL^2\paren{1 + \frac{2}{r^{3/2}_*(1-\sqrt{1-\lambda})}\left(1+\frac{\sqrt{n}GD_\sX}{L}\right)}\sqrt{T} \\
	& \qquad \qquad + C \sqrt{T+1},
	\end{align*}
\end{theorem}

\begin{proof} \an{By Lemma~\ref{lem:disagree}, the averaging policy \eqref{eqn:Markov_policy} satisfies
\begin{align*}
	\sum_{i=1}^n \|\zb_i(t) - \bzb(t)\|^2
	 \le
	 \frac{n L^2}{r^3_*(1-\sqrt{1-\lambda})^2}.
\end{align*}
	Hence, by Jensen's inequality,
	\begin{align*}
		\sum_{i=1}^n \|\zb_i(t) - \bzb(t)\|&\le\sqrt{n\sum_{i=1}^n \|\zb_i(t) - \bzb(t)\|^2} \\
	&\le  \frac{nL}{r^{3/2}_*(1-\sqrt{1-\lambda})}.
\end{align*}
	Therefore, the conditions of Corollary~\ref{cor:regret_bound} hold} with
	\begin{align*}
		 K =   \frac{nL}{r^{3/2}_*(1-\sqrt{1-\lambda})},
	\end{align*}
\an{and the stated} result follows.~$\square$
\end{proof}
\sml{This shows that, for any fixed communication network $\Gc$ satisfying Assumption \ref{assume:M},
the worst-case regret is bounded by $O(\sqrt{T})$. The constants also capture the dependence on the algebraic connectivity of the network via the spectral gap $\lambda$, as well as on the network size $n$.}

\section{\textbf{ODA-PS} and its Regret Bound}\label{sec:algo-regret2}
We now introduce another decentralized online optimization algorithm which uses
the push-sum communication protocol
for its dual update rule \eqref{eqn:gdyn1}.
We refer to this algorithm as \textbf{ODA-PS} (Online Dual Averaing with Push-Sum based communication).
\textbf{ODA-PS} uses the network model (G2) for its communication.

\subsection{\textbf{ODA-PS}}
For \textbf{ODA-PS}, each agent $i$ maintains an additional scalar sequence $\{w_i(t)\}_{t=1}^{\infty} \subset \mathbb{R}$.
Then, this algorithm particularizes the update rule in (\ref{eqn:gdyn1})-(\ref{eqn:gdyn2}) as
\begin{subequations}
\begin{align}
w_i(t+1) =~& \sum_{j=1}^n [A(t)]_{ij} w_j(t)\label{eqn:algo1}\\
z_i^k(t+1) =~& n\delta_i^ku_i(t) + \sum_{j=1}^n [A(t)]_{ij} z_j^k(t) ,~k\in [n]\label{eqn:algo2}\\
\xb_i(t+1) =~& \Pi^\psi_{\sX^n}\paren{\frac{\zb_i(t+1)}{w_i(t+1)},\a(t)}\label{eqn:algo3}
\end{align}
\end{subequations}
where the weight matrix $A(t)$
is defined by the out-degrees of the in-neighbors, i.e.,
\begin{equation}\label{eqn:At}
[A(t)]_{ij} = \left\{
\begin{array}{ll}
1/d_j(t) & \text{whenever } j \in \Nin(t)\\
0 & \text{otherwise. }
\end{array}
\right.
\end{equation}
The matrix $A(t)$ is column stochastic by construction.

Note that the above update rules are based on a simple broadcast communication.
Each agent $i$ broadcasts (or \textit{pushes}) the quantities
$w_i(t)/d_i(t)$ and  $\zb_i(t)/d_i(t)$ to all of the nodes in its
out-neighborhood $\Nc_i^{\textrm{out}}(t)$.
Then, in \eqref{eqn:algo1}-\eqref{eqn:algo2}
each agent simply \textit{sums} all the received messages
 to obtain $w_i(t+1)$ and $\zb_i(t+1)$. The update
rule \eqref{eqn:algo3} can be executed locally.
\sml{Unlike \textbf{ODA-C}, the averaging matrix $A(t)$ in \textbf{ODA-PS} does not require symmetry due to this broadcast-based nature of the push-sum protocol. However, the asymmetry requires uniformity of the positive weights $r_i$ across all agents (cf.~Eq.~\eqref{eqn:gdyn1}). Here we simply use $r_i = 1/n$.
}

To complete the description of the algorithm, we must specify the update policies $\{u_i(t)\}$.
As in \textbf{ODA-C}, we assume that the signal agent $i$ gets from the environment at
time $t$ is simply the $i$-th coordinate of the gradient of $f_t$ at the
agent’s primal variable $\xb_i(t)$. Thus, we define:
\begin{equation}\label{eqn:policy1}
u_i(t) = \la \nabla f_t(\xb_i(t)),\eb_i\ra, ~ i \in [n],~t\ge 0,
\end{equation}
i.e., the update performed by agent $i$ at time $t$ is the simply the $i$-th coordinate of the gradient of $f_t$ at the agent's primal variable $\xb_i(t)$.

We assume that each agent $i$ initializes its updates with $w_i(0)=1$ and $\zb_i(0)=0$, while $u_i(0)$ can be any arbitrary value in $\mathsf{X}$.
We also recall that the local action of agent $i$ at time $t$ is given by the $i$th coordinate of $\xb_i(t)$, i.e.,
\[
x^i(t) = x_i^i(t).
\]
For notational convenience, let us denote the products of the weight matrices $A(t), \ldots, A(s)$ by $A(t:s)$, i.e.,
\[
A(t:s) \triangleq A(t) \cdots A(s) \quad \text{for all } t \ge s \ge 0.
\]
Also, we denote
\[
A(t-1:t) \triangleq I, \quad \text{for all } t \ge 1.
\]

\subsection{Regret of \textbf{ODA-PS} with local gradient signals}
For the regret analysis, we first study the dynamics of the dual iterates $\zb_i(t)$ and its ``mean field'' $\bar{\zb}(t)$ in the following lemma.
We remind that $\bar{\zb}(t) = (\bar{z}^1(t),\ldots,\bar{z}^n(t))$ and
\begin{align*}
\zb^k(t)=(z^k_1(t),\ldots,z^k_n(t)),~k\in[n].
\end{align*}

\begin{lemma}\label{lem:zbar}
Let $\zb_i(0) = 0$ for all $i \in \Vc$.
\begin{itemize}
\item[(a)]
The weighted sum
\[
\bar{\zb}(t) =  \frac{1}{n}\sum_{i=1}^n \zb_i(t)
\]
evolves according to the linear dynamics
\begin{equation*}
\bar{\zb}(t+1) = \bar{\zb}(t) + \ub(t),
\end{equation*}
where $\ub(t) = (u_1(t),\ldots,u_n(t))$.
\item[(b)] For any $i,k\in [n]$, the iterates in (\ref{eqn:algo2}) evolve according to the following dynamics
\begin{equation*}
z_i^k(t) =n \sum_{s=0}^{t-1} [A(t-1:s+1)]_{ik} u_k(s).
\end{equation*}
\end{itemize}
\end{lemma}
\begin{proof}
\begin{itemize}
\item[(a)] From relation (\ref{eqn:algo2}), we have for all $k \in [n]$
\begin{align*}
\bar{z}^k(t+1) = &~\frac{1}{n}\sum_{i=1}^nz_i^k(t+1)\\
=&~\frac{1}{n}\sum_{i=1}^n \left[n\delta_i^ku_i(t) + \sum_{j=1}^n [A(t)]_{ij}z_j^k(t)\right]\\
=&~u_k(t) + \frac{1}{n}\sum_{j=1}^n z_j^k(t)\sum_{i=1}^n[A(t)]_{ij}\\
=&~u_k(t) + \bar{z}^k(t),
\end{align*}
where the last equality follows from the column-stochasticity of the matrix $A(t)$.
The desired result follows by stacking up the scalar relation above over $k$.
\item[(b)]
By stacking up the equation (\ref{eqn:algo2}) over $i$, we have for all $t \ge 1$ and $k \in [n]$
\begin{align*}
\zb^k(t+1) = A(t)\zb^k(t) + nu_k(t)\eb_k.
\end{align*}
By unrolling this equation from time 0 to $t$, we obtain
\begin{align*}
\zb^k(t) = &~A(t-1:0)\zb^k(0) \\
&~+ n\sum_{s=0}^{t-1}u_k(s) A(t-1:s+1)\eb_k\\
= &~ n\sum_{s=0}^{t-1}u_k(s) A(t-1:s+1)\eb_k,
\end{align*}
where the equalities follows from $A(t-1:t) = I$ and the initial condition $\zb_i(0) = 0$ for all $i \in \Vc$.
We get the desired result by taking the $i$-th component of this vector. $\square$
\end{itemize}
\end{proof}
\sml{Lemma \ref{lem:zbar} tells us that the vector $\bar\zb(t)$ acts as a ``mean field'' of the dual iterates $\zb_i(t)$.
Also, if we define
\[
\bar\xb(t+1) \triangleq \Pi^\psi_{\sX^n}\paren{\bar{\zb}(t+1),\a(t)},
\]
then from Lemma \ref{lem:zbar}(a) we can see that
\[
\bar\xb(t+1) \triangleq \Pi^\psi_{\sX^n}\paren{\sum_{s=1}^t \ub(s),\a(t)},
\]
which coincides with relation \eqref{eq:xu} in Theorem \ref{thm:main}.
}

We now particularize the bound in Theorem \ref{thm:main} in this scenario under the additional assumption on the Lipschitz continuous gradients (Assumption \ref{assume:Lipgrad} in Section \ref{sec:algo-regret}).
\begin{theorem}\label{thm:4}
Under Assumptions \ref{assume:Lipschitz}-\ref{assume:Lipgrad},
the regret of the algorithm (\ref{eqn:algo1})-(\ref{eqn:algo3})
with the local update $u_i(t)$ of agent $i$ computed according to (\ref{eqn:policy1})
can be upper-bounded as follows: for all $T\ge 1$,
\begin{align*}
R(T) \le~& \frac{nL^2}{2}\sum_{t=1}^T \a(t-1)+ \frac{C}{\a(T)}\nonumber\\
~&+ \paren{L + \sqrt{n}GD_{X}}\sum_{t=1}^T \a(t-1)\sum_{i=1}^n\norm{\frac{\zb_i(t)}{w_i(t)}-\bar{\zb}(t)}.
\end{align*}
\end{theorem}
\begin{proof}
Since the definition of $\ub(t)$ in \textbf{ODA-PS} (cf. Eq. \eqref{eqn:policy1}) coincides with that
in \textbf{ODA-C} (cf. Eq. \eqref{eqn:gradient_estimate}), we can reuse all the derivations in the proof of Theorem \ref{thm:local_grad_signals} except for the network-wide disagreement term:
	\begin{align} \label{eqn:thm222}
		&\| \bxb(t) - \xb_i(t) \| \nonumber\\
		&\quad = \left\| \Pi^\psi_{\sX^n}(\bzb(t),\alpha(t-1)) - \Pi^\psi_{\sX^n}\paren{\frac{\zb_i(t)}{w_i(t)},\alpha(t-1)} \right\| \nonumber \\
		&\quad \le \alpha(t-1) \norm{ \bzb(t) - \frac{\zb_i(t)}{w_i(t)} },
 	\end{align}
where the last inequality follows from
the $\alpha$-Lipschitzian property of the map $\zb \mapsto \Pi^\psi_{\sX^n}(\zb,\alpha)$ \cite[Lemma~1]{Nesterov_dual_averaging}.
$\square$
\end{proof}
\sml{This bound tells us that the regret $R(T)$ will be sublinear in $T$ with proper choice of the step size $\a(t)$ if the network-wide disagreement term behaves nicely. Note that we can also make use of Corollary \ref{cor:regret_bound}
here if we can show
	\begin{align*}
		\sum^n_{i=1} \norm{ \bzb(t) - \frac{\zb_i(t)}{w_i(t)} } \le K,
	\end{align*}
for some constant $K>0$.
}

\subsection{Full regret analysis}
We now show that the network-wide disagreement term in Theorem \ref{thm:4} is indeed upper-bounded by some constant. For doing this, we first restate a lemma from \cite{NO}. 
\begin{lemma}\label{lem:NO}
Let the graph sequence $\{\Gc(t)\}$ be $B$-strongly connected. Then the following statements are valid.
\begin{itemize}
\item[(a)] There is a sequence $\{\phi(t)\} \subseteq \mathbb{R}^n$ of stochastic vectors such that
the matrix difference $A(t:s)-\phi(t)\1b'$ for $t \ge s$ decays geometrically, i.e., for all $i, j \in [n]$.
\[
\abs{[A(t:s)]_{ij}-\phi_i(t)} \le \beta\t^{t-s} \quad \text{for all } t \ge s \ge 0,
\]
where we can always choose
\[
\beta = 4, \quad \t = (1-1/n^{nB})^{1/B}.
\]
If in addition each $\Gc(t)$ is regular, we may choose
\[
\beta = 2\sqrt{2},\quad \t = (1-1/4n^3)^{1/B},
\]
or
\[
\beta = \sqrt{2},\quad \t = \max_{t\ge 0} \sigma_2(A(t)),
\]
whenever $\sup_{t\ge 0}\sigma_2(A(t))<1$.
\item[(b)]
The quantity
\[
\gamma = \inf_{t \ge 0} \paren{\min_{1\le i\le n}[A(t:0)\1b]_i}
\]
satisfies
\[
\gamma \ge \frac{1}{n^{nB}}.
\]
Moreover, if the graphs $\Gc(t)$ are regular, we have $\gamma = 1$.
\end{itemize}
\end{lemma}
The next lemma provides an upper-bound for
$\sum_{i=1}^n \left\|\frac{\zb_i(t)}{w_i(t)}-\bar{\zb}(t)\right\|^2$.
\begin{lemma}\label{lem:network}
Let the sequences $\{\zb_i(t)\}$ and $\{w_i(t)\}$ be generated according to the algorithm (\ref{eqn:algo1})-(\ref{eqn:algo2}).
Recall that $\bar{\zb}(t) = \frac{1}{n}\sum_{i=1}^n \zb_i(t)$.
Then, we have for all $t \ge 1$,
\begin{equation*}
\sum_{i=1}^n\norm{\frac{\zb_i(t)}{w_i(t)}-\bar{\zb}(t)}^2 \le  n^2\paren{\frac{2\beta L}{\gamma\t(\t-1)}}^2,
\end{equation*}
where the constants $\beta$, $\gamma$ and $\t$ are as defined in Lemma \ref{lem:NO}.
\end{lemma}
\begin{proof}
From the definitions of $\zb_i(t)$, $\bar{\zb}(t)$ and $\zb^k(t)$, we have
\begin{align}\label{eqn:equ}
\sum_{i=1}^n\norm{\frac{\zb_i(t)}{w_i(t)}-\bar{\zb}(t)}^2
= \sum_{i=1}^n \sum_{k=1}^n \paren{\frac{z_i^k(t)}{w_i(t)}-\bar{z}^k(t)}^2.
\end{align}
Thus, we can upper-bound the quantity on the right-hand side.

By inspecting equation (\ref{eqn:algo1}), it is easy to see that for any $i \in \Vc$ and $t \ge 1$, we have
\[
w_i(t) = \sum_{\ell=1}^n[A(t-1:0)]_{i\ell} w_i(0) = \sum_{\ell=1}^n[A(t-1:0)]_{i\ell}.
\]
From this and Lemma \ref{lem:zbar}, we have the following chain of relations:
\begin{align}\label{eqn:lem41}
& \frac{z_i^k(t)}{w_i(t)}-\bar{z}^k(t) \nonumber\\
~= &\frac{n\sum_{s=0}^{t-1}[A(t-1:s+1)]_{ik}u_k(s)}{\sum_{\ell=1}^n[A(t-1:0)]_{i\ell}} - \sum_{s=0}^{t-1}u_k(s)\nonumber\\
~= &\sum_{s=0}^{t-1}u_k(s)\frac{\sum_{\ell=1}^n[A(t-1:s+1)]_{ik}-\sum_{\ell=1}^n[A(t-1:0)]_{i\ell}}{\sum_{\ell=1}^n[A(t-1:0)]_{i\ell}}\nonumber\\
~\le &\sum_{s=0}^{t-1}u_k(s) \Bigg(\frac{\sum_{\ell=1}^n\left([A(t-1:s+1)]_{ik}-\phi_i(t-1)\right)}{\sum_{\ell=1}^n[A(t-1:0)]_{i\ell}} \nonumber\\
~& + \frac{\sum_{\ell=1}^n\left(\phi_i(t-1)-[A(t-1:0)]_{i\ell}\right)}{\sum_{\ell=1}^n[A(t-1:0)]_{i\ell}}\Bigg) \nonumber\\
~\le &\sum_{s=0}^{t-1}u_k(s) \frac{\beta\t^{t-s-2}+\beta\theta^{t-1}}{\gamma},
\end{align}
where the inequalities follow from adding and subtracting $\phi_i(t-1)$
and from Lemma \ref{lem:NO}.
From relation (\ref{eqn:policy1}), we have
\[
|u_k(s)|^2 = |\la \nabla f_s(\xb_k(s)),\eb_k\ra|^2 \le \|\nabla f_s(\xb_k(s))\|^2\le L^2.
\]
Combining this and the fact that $\beta\t^{t-s-2}\ge \beta\theta^{t-1}$ for all $s = 0,\ldots,t-1$, we further have
\begin{align*}
\abs{\frac{z_i^k(t)}{w_i(t)}-\bar{z}^k(t)}
\le  \sum_{s=0}^{t-1}|u_k(s)| \frac{2\beta\t^{t-s-2}}{\gamma}
\le  \frac{2\beta L}{\gamma\t(\t-1)}.
\end{align*}
Substituting this estimate in relation \eqref{eqn:equ}, we get the desired result.
$\square$
\end{proof}
By combining Theorem \ref{thm:4} and Lemma \ref{lem:network}, we can now provide the regret bound of \textbf{ODA-PS}:
\begin{theorem}\label{thm:linear_ave2} \an{Let Assumptions~\ref{assume:Lipschitz}--\ref{assume:Lipgrad} hold.
With the choice $\alpha(t) = \frac{1}{\sqrt{t+1}}$ for all $t\ge0$,
and under  the policy \eqref{eqn:policy1}, the distributed algorithm \textbf{ODA-PS} achieves the following regret:}
\begin{align*}
R(T) &\le nL^2\paren{1+\paren{1 + \frac{\sqrt{n}GD_{\mathsf{X}}}{L}}\frac{4\beta \sqrt{n}}{\gamma\t(1-\t)}}\sqrt{T}\nonumber\\
& \qquad \qquad \qquad + C\sqrt{T+1}\nonumber,
\end{align*}
where the constants $\beta$, $\gamma$ and $\t$ are as defined in Lemma \ref{lem:NO}.
\end{theorem}

\begin{proof}
By Jensen's inequality, we have
\begin{align*}
\sum_{i=1}^n \norm{\frac{\zb_i(t)}{w_i(t)}-\bar\zb(t)} \le \sqrt{n\sum_{i=1}^n \norm{\frac{\zb_i(t)}{w_i(t)}-\bar\zb(t)}^2}.
\end{align*}
Hence, using Lemma \ref{lem:network}, we can estimate the network-wide disagreement term as follows:
\begin{align*}
\sum_{i=1}^n \norm{\frac{\zb_i(t)}{w_i(t)}-\bar\zb(t)} & \le \sqrt{n^3\paren{\frac{2\beta L}{\gamma\t(\t-1)}}^2} \\
& = \frac{2\beta n^{3/2}L}{\gamma\t(\t-1)}.
\end{align*}
Thus, the conditions of Corollary \ref{cor:regret_bound} with this modified network-wide agreement hold with
\[
K = n\sqrt{n}\frac{2\beta L}{\gamma\t(\t-1)}.
\]
and the stated result follows.~$\square$
\end{proof}
\sml{The bound shows that, for any time-varying sequence of $B$-strongly connected digraphs, the worst-case regret of \textbf{ODA-PS}
is of order $O(\sqrt{T})$. The constants also capture the dependence on the properties of the underlying network, i.e., the number of nodes $n$ and as well as the connectivity period $B$.}
\vspace{-0.1in}

\section{Simulation Results\label{sec:sim}}

Consider the problem of estimating some target vector $\xb \in \mathbb{R}^p$
using measurements from a network of $n$ 
sensors.
Each sensor $i$ is in charge of estimating a subvector $\xb_i \in \mathbb{R}^{p_i}$ of $\xb$,
where $p_i \ll p$ and $p = \sum_{i=1}^n p_i$ is some very large number.
An example includes the localization of multiple targets,
where in this case $\xb \in \mathbb{R}^p$ becomes a stacked vector of all target locations.
When there are a number of spatially dispersed targets,
we can certainly benefit from distributed 
sensing.

The sensors are assumed to have a linear model of $r(\xb) = A\xb$,
where $A \in \mathbb{R}^{m\times p}$ and $m< p$.\footnote{{  Although target localization is usually formulated as a nonlinear estimation problem \cite{sourceloc}, for considerations of simplicity one often employs a linearized model using a first-order Taylor expansion around the measurements; see, e.g., \cite{Kleeman1995,Ponda2008}.}}
At each time $t$, each sensor $i \in \Vc$
estimates its portion $\xb_i(t)\in \mathbb{R}^{p_i}$ of the target vector $\xb \in \mathbb{R}^p$,
and then takes a measurement $q_t^i \in \mathbb{R}^{m_i}$,
which is corrupted by observation error and possibly by modeling error.
We assume all sources of errors can be represented as an additive noise, i.e.,
\[
\mathbf{q}_t = A\xb(t) + \zeta_t,
\]
where $\mathbf{q}_t \in \mathbb{R}^m$  with $m = \sum_{i=1}^n m_i$ is a stacked vector of all $q_t^i$'s
and $\zeta_t \sim N(0,P)$, where $P$ is the noise covariance matrix.

The regret is computed with respect to the least-squares estimate of the target locations at time $T$, i.e.,
\[
\hat{\mathbf{x}} = \argmin_{\mathbf{x}\in \mathsf{X}^p}\sum_{t=1}^T f_t(\xb),
\]
where
$
f_t(\xb) = \frac{1}{2}\|A\xb - \mathbf{q}_t\|^2.
$
and we set $\mathsf{X} \in [-20,~20]$.

For \textbf{ODA-C}, we experiment with a $n=5$ node cycle graph whose communication topology is given as:
\[
1 \leftrightarrow 2 \leftrightarrow 3 \leftrightarrow 4\leftrightarrow 5 \leftrightarrow 1
\]
We set $r_i = 1/5$, $M_{ii} = 1/2$ for all $i$, and $M_{ij} = 1/4$ if $i \leftrightarrow j$.
For \textbf{ODA-PS}, we experiment with a time-varying sequence of digraphs with $n=5$ nodes whose communication topology is changing periodically with period $3$. The graph sequence is, therefore, $3$-strongly connected. In Figure \ref{fig:graph}, we depict the repetition of the 3 corresponding graphs.
The averaging matrices $A(t)$ (cf. Eq. \eqref{eqn:At}) can be determined accordingly.
We ran our algorithms once for each $T \in [1000]$. That is, for a given $T$, the iterates in the algorithms are updated from $t = 1$ to $t = T$. 
We used step size $\a(t) = \frac{1}{\sqrt{t+1}}$ for both algorithms.

\begin{figure}[t]
\begin{center}
\includegraphics[scale=0.37]{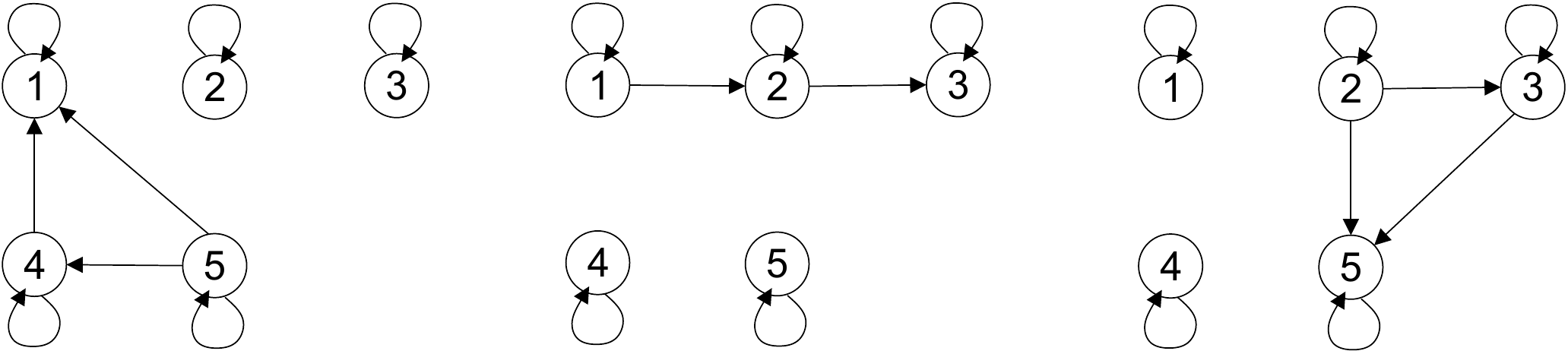}
\end{center}
\caption{Time-varying communication topology changing in cycle of three used for \textbf{ODA-PS}\label{fig:graph}}
\end{figure}

\begin{figure}[t]
\begin{center}
\includegraphics[scale=0.23]{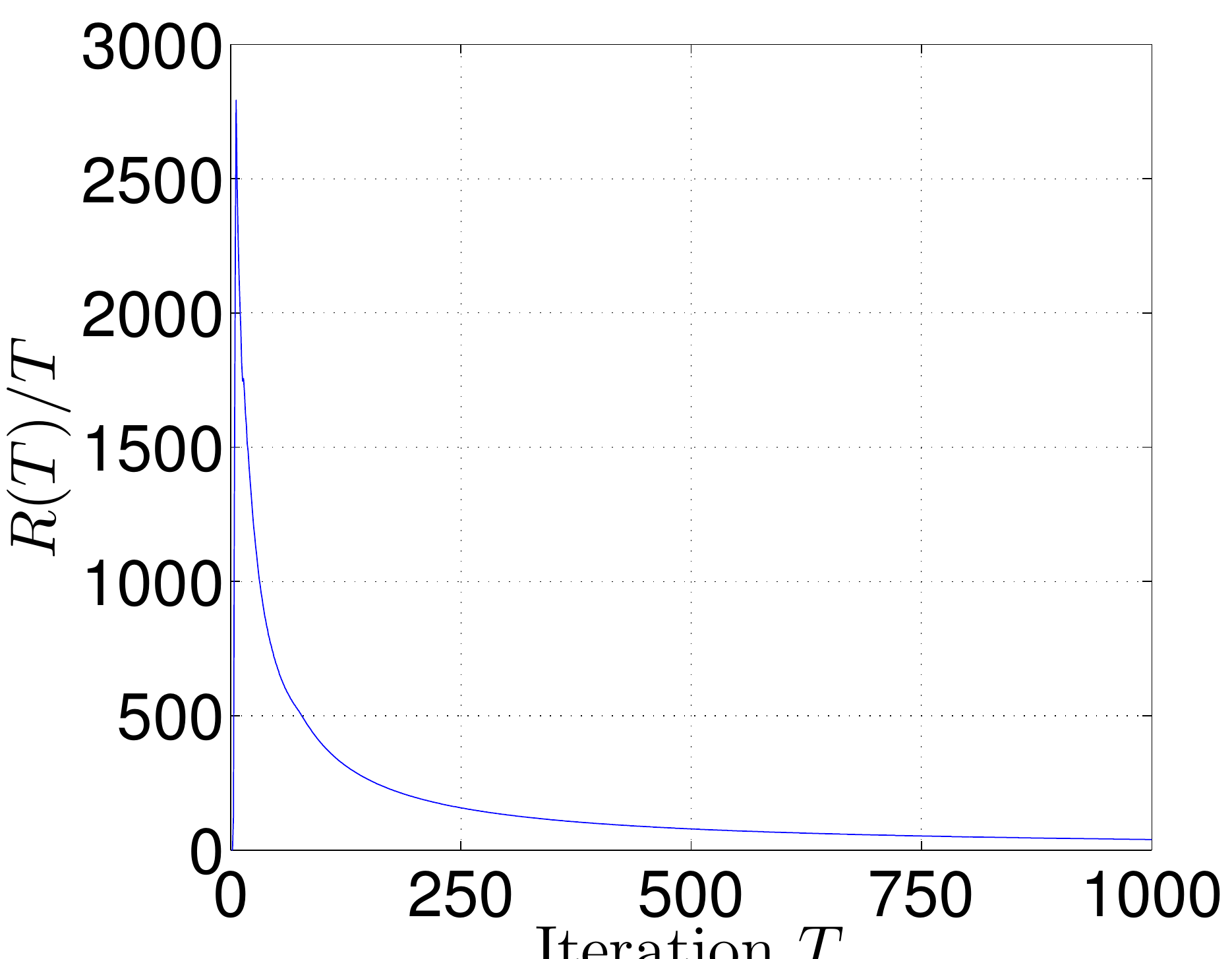}
\includegraphics[scale=0.23]{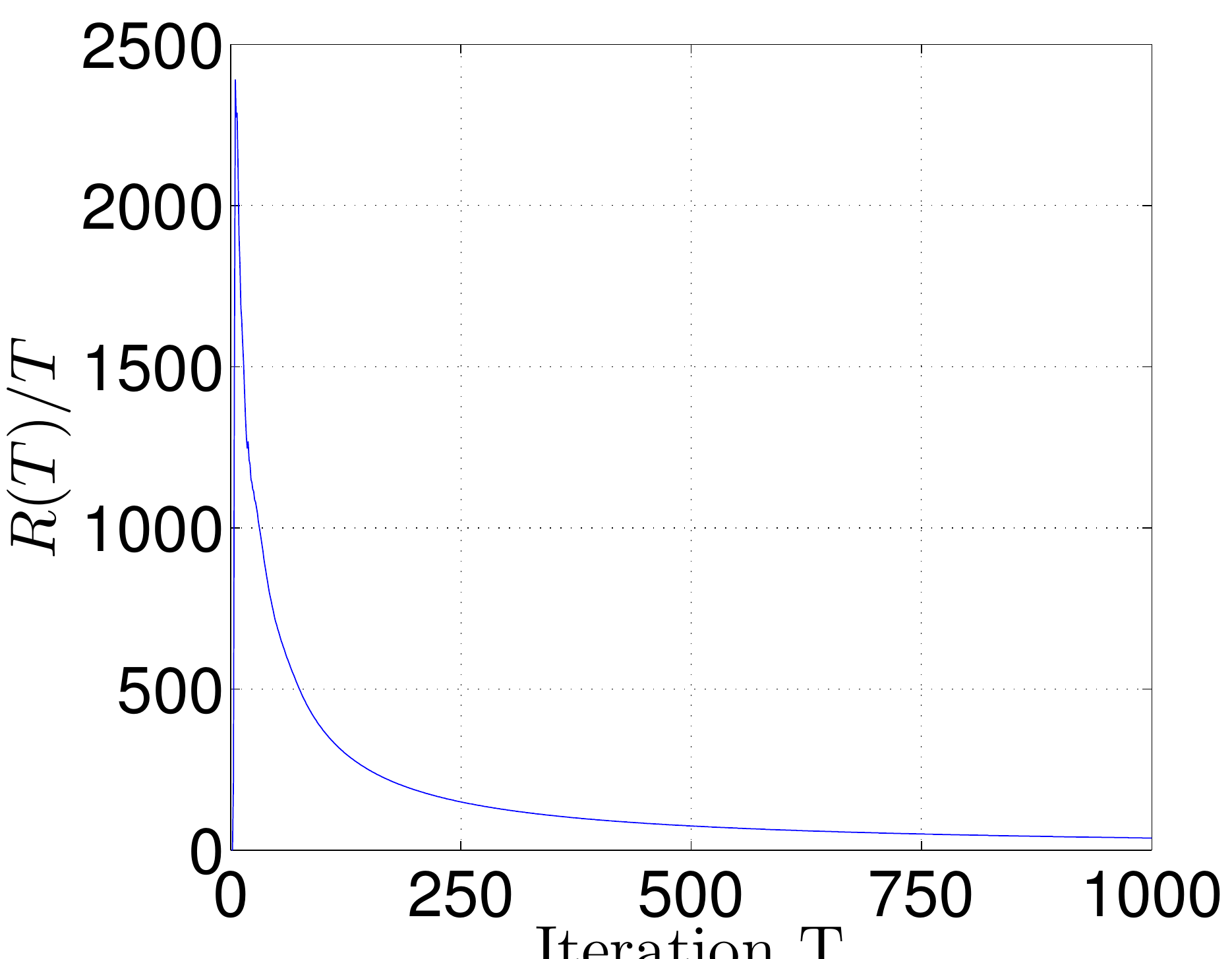}
\end{center}
\caption{The Average Regret $R(T)/T$ vs. Iterations for Online Distributed Active Sensing using \textbf{ODA-C} (left) and \textbf{ODA-PS} (right) \label{fig:Sensing}}
\end{figure}
In Figure \ref{fig:Sensing}, we depict the average regret $R(T)/T$ over time $T$ of the distributed sensing problem when \textbf{ODA-C} and \textbf{ODA-PS} are used, respectively.
It shows that the regret is sublinear for both algorithms and the average $R(T)/T$ goes to zero as the time increases.

\section{Conclusion}\label{sec:conclusion}
We have studied an online optimization problem in a multiagent network.
We proposed two decentralized variants of Nesterov's primal-dual algorithm, namely,
\textbf{ODA-C} using circulation-based dynamics for time-invariant networks
and \textbf{ODA-PS} using broadcast-based push-sum protocol for time-varying networks.
We have established a generic regret bound and provided its refinements for certain information exchange policies.
The regret is shown to grow as $O(\sqrt{T})$
when the step size is $\alpha(t)=1/\sqrt{t+1}$.
For \textbf{ODA-C}, the bound is valid for a static connectivity graph and a row-stochastic matrix of weights $M = [M_{ij}]$ which is reversible with respect to a strictly positive probability vector $\rb$.
For \textbf{ODA-PS}, the bound is valid for a uniformly strongly connected sequence of digraphs and column-stochastic matrices of weights $A(t)$ whose components are based on the out-degrees of neighbors.
Simulation results on a sensor network exhibit the desired theoretical properties of the two algorithms.
\vspace{-0.1in}


\end{document}